\documentclass[12pt,a4paper]{article}

\RequirePackage{etex}
\usepackage[utf8]{inputenc}
\usepackage{amsmath, amssymb, amsthm, amsfonts, mathrsfs}
\usepackage{bbm, commath}
\usepackage{multicol}
\usepackage{graphicx}
\usepackage{tikz-network}
\usetikzlibrary {positioning,patterns, calc}
\usetikzlibrary{graphs,graphs.standard,quotes}
\usepackage{authblk}
\usepackage{enumitem}
\usepackage[colorlinks=true,linkcolor=black,citecolor=black]{hyperref}
\usepackage{tabularx}
\usepackage{array}

\newtheorem{thm}{Theorem}

\newtheorem{prop}{Proposition}
\newtheorem{lem}{Lemma}
\newtheorem{cor}{Corollary}

\def \e {{\bf e}}

\def \e {{\mathbf e}}

\title{\it A Characterization of State Transfer on Double Subdivided Stars}
\author[1]{Sarojini Mohapatra}
\author[2]{Hiranmoy Pal}
\affil[1,2]{National Institute of Technology Rourkela, India-769008.}
\date{\today}

\begin{document}
	
	\maketitle

	%%%%%%%%%%% Abstract %%%%%%%%%%%
	
		\begin{abstract}
 A subdivided star $SK_{1,l}$ is obtained by identifying exactly one pendant vertex from $l$ copies of the path $P_3.$ This study is on the existence of quantum state transfer on double subdivided star $T_{l,m}$ which is a pair of subdivided stars $SK_{1,l}$ and $SK_{1,m}$ joined by an edge to the respective coalescence vertices. Using the Galois group of the characteristic polynomial of $T_{l,m},$ we analyze the linear independence of its eigenvalues which uncovers no perfect state transfer in double subdivided stars when considering the adjacency matrix as the Hamiltonian of corresponding quantum system. Then we establish a complete characterization on double subdivided stars exhibiting pretty good state transfer.
	\end{abstract}

	%%%%%%%%%%% Introduction %%%%%%%%%%%
	\newpage
	\section{Introduction}
The transfer of quantum states between two different locations in a quantum network plays an important role in quantum information processing. Let a quantum network of $n$ interacting qubits be modeled by a graph $G$ where the vertices correspond to the qubits and the edges represent the interactions between qubits. Transfer of state among such qubits can be described using the continuous-time quantum walk operator acting on the characteristic vectors of the vertices. If the network admits a transfer of quantum state between two qubits without any loss of information, then this phenomenon is called perfect state transfer (PST). The main objective here is to identify quantum networks which enable high probability state transfer between qubits. We consider state transfer with respect to the adjacency matrix of a graph $G$ having the vertex set $\{a_1,a_2,\ldots,a_n\}.$ The adjacency matrix $A=[a_{jk}]$ is the $n\times n$ matrix having $a_{jk}=1$ if there is an edge between $a_j$ and $a_k$, otherwise $a_{jk}=0$. The continuous-time quantum walk on $G$ relative to the adjacency matrix $A$ is defined by \[U(t):=\exp{(itA)},~\text{where}~t\in\mathbb{R}~\text{and}~i=\sqrt{-1}.\]
Farhi and Gutmann \cite{farhi} first used the method of continuous-time quantum walks in analysing various quantum transportation phenomena. One can observe that the transition matrix $U(t)$ is symmetric as well as unitary. The square of the absolute value of $(a,b)$-th entry of $U(t)$ provide the probability of state transfer from site $a$ to site $b$ after time $t$. Suppose all the distinct eigenvalues of $A$ are $\theta_1,\theta_2,\ldots,\theta_d$. Let $E_{\theta_j}$ denote the orthogonal projection onto the eigenspace corresponding to $\theta_j.$  The spectral decomposition of the transition matrix $U(t)$ can be evaluated as
\[U(t)=\sum_{j=1}^{d}\exp{(it\theta_j)}E_{\theta_j}.\] 
Let $\e_a$ denote the characteristic vector corresponding to a vertex $a$ of $G$. The eigenvalue support of $a$ defined by $\sigma_a=\{\theta_j:E_{\theta_j}\e_a\neq0\}.$ The graph $G$ is said to exhibit PST between a pair of distinct vertices $a$ and $b$  if there exists $\tau\in\mathbb{R}$ such that 
\begin{equation}\label{equ1}
    U(\tau)\e_a=\gamma \e_b,~\text{for some}~\gamma\in\mathbb{C}.\end{equation}
It is now evident that the existence of PST between $a$ and $b$ depends only on the eigenvalues in the support $\sigma_a$ and the corresponding orthogonal projections. PST in quantum communication networks was first introduced by Bose in \cite{bose}. 
 There it shows that PST occurs between the end vertices of the path $P_2$ on two vertices.  
In \cite{chr1}, Christandl et al. proved that PST occurs between the end vertices of a path $P_n$ on $n$ vertices if and only if $n=2$, $3$. Remarkably, Ba\v{s}i\'{c} \cite{mil4} established a complete characterization of integral circulant graphs having PST. The existence of PST in several well-known families of graphs and their products are also investigated in \cite{ack1, ange1, pal9, che, cou7, pal1, pal2}, etc. Later, Coutinho et al. \cite{cou5} showed that there is no PST in a graph $G$ between two cut vertices $a$ and $b$ that are connected only by the path $P_2$ or $P_3$, unless the graph $G$ is itself $P_2$ or $P_3.$ This infers that the double subdivided star $T_{l,m}$ does not exhibit PST between the coalescence vertices of degree $l$ and $m$ for all positive integers $l$ and $m$. Here we show that $T_{l,m}$ does not exhibit PST between any pair of vertices for all such cases using the linear independence of the eigenvalues of $T_{l,m}$ in Section \ref{s3}.

In case $a=b$ in \eqref{equ1}, the graph $G$ is said to be periodic at the vertex $a$ with period $\tau$. If $G$ is periodic at every vertex with the same period, then it is called a periodic graph. It is well known that if there is PST between a pair of vertices $a$ and $b$ at time $\tau$ then $G$ is periodic at both $a$ and $b$ with period $2\tau$. Therefore periodicity at the vertex $a$ is necessary for the existence of PST from $a$. 
In what follows, we find that if $G$ is periodic at a vertex then it must satisfy the following ratio condition.
 \begin{thm}\label{p6}\cite{god3}
    Suppose a graph $G$ is periodic at vertex $a$. If $\theta_k,\theta_l,\theta_r, \theta_s$ are eigenvalues in the support of $a$ and $\theta_r\neq \theta_s$, then
    \[\dfrac{\theta_k- \theta_l}{\theta_r-\theta_s}\in\mathbb{Q}.\]
\end{thm}
The existence of PST in graphs is a rare phenomena as observed in \cite{god2}, and consequently, the notion of pretty good state transfer (PGST) was introduced in \cite{god1, vin}. A graph $G$ is said to exhibit PGST between a pair of distinct vertices $a$ and $b$ if there exists a sequence $\tau_k$ of real numbers such that 
\[\lim_{k\to\infty}U(\tau_k)\e_a=\gamma\e_b,~\text{for some}~~ \gamma\in \mathbb{C}.\] In \cite{god4}, Godsil et al. showed that there is PGST between the end vertices of $P_n$ if and only if $n+1=2^t~\text{or}~p~\text{or}~2p$, for some positive integer $t$ and odd prime $p$. Moreover, if there is PGST between the end vertices of $P_n$, then it occurs between the vertices $a$ and $n+1-a$ as well, whenever $a\neq (n+1)/2.$ Further investigation is done in \cite{cou3} to determine infinite family of paths admitting PGST between a pair of internal vertices, where there is no PGST between the end vertices. Among other trees, PGST is investigated on double star \cite{fan}, $1$-sum of stars \cite{hou}, etc. Pal et al. \cite{pal4} showed that a cycle $C_n$ and its complement $\overline{C}_n$ admit PGST if and only if $n$ is a power of $2$, and it occurs between every pair of antipodal vertices. It is worth noting that PGST is not monogamous unlike PST as argued in \cite[Example 4.1]{pal5}. More results on PGST can be found in \cite{cou6, eis, pal6, pal7, bom1}, etc. Here we investigate the existence of PGST on double subdivided stars. A subdivided star with $l$ branches, denoted by $SK_{1,l},$ is obtained by identifying exactly one pendant vertex from $l$ copies of the path $P_3.$ A double subdivided star is formed by joining the coalescence vertices of a pair of subdivided stars $SK_{1,l}$ and $SK_{1,m}$ by an additional edge, and the resulting graph is denoted by $T_{l,m}.$ We analyze the linear independence of the eigenvalues of $T_{l,m}$ in Section \ref{s2} and then, in Section \ref{s3}, the existence of PGST in $T_{l,m}$ is investigated.  

A pair of vertices $a$ and $b$ in a graph $G$ are called strongly cospectral if $E_{\theta_j}\e_a=\pm E_{\theta_j}\e_b,$ for all eigenvalues $\theta_j$. Next we observe that strong cospectrality is necessary for the existence of PGST between a pair of vertices.
\begin{lem}\cite{god1}\label{l4}
If a graph $G$ exhibits pretty good state transfer between a pair of vertices $a$ and $b$, then they are strongly cospectral. 	
\end{lem}
 If $P$ is a matrix of an automorphism of $G$ with adjacency matrix $A$  then $P$ commutes with $A$. Since the transition matrix $U(t)$ is a polynomial in $A$, the matrices $P$ and $U(t)$ commute as well. Therefore, if $G$ allows PGST between $a$ and $b,$ then each automorphism fixing $a$ must fix $b$.
 We use the following Kronecker approximation theorem on simultaneous approximation in characterizing double subdivided stars having PGST.
\begin{thm}\label{t1} \cite{apo} 
Let $\alpha_1,\alpha_2,\ldots,\alpha_l$ be arbitary real numbers. If $1,\theta_1,\ldots,\theta_l $ are real, algebraic numbers linearly independent over $\mathbb{Q}$, then for $\epsilon>0$, there exists $q\in\mathbb{Z}$ and $p_1,p_2,\ldots,p_l\in\mathbb{Z}$ such that 
	\[|q\theta_j-p_j-\alpha_j|<\epsilon.\]
\end{thm}
Now we recall few results on the spectra of graphs. Let $G$ be a graph having distinct eigenvalues $\theta_1, \theta_2, \ldots , \theta_d$ with
multiplicities $k_1, k_2,\ldots,k_d$, respectively. We denote the spectrum of $G$ as, $\theta_1^{k_1}, \theta_2^{k_2}, \ldots , \theta_d^{k_d}$. In case $\theta_i$ is a simple eigenvalue, we omit the power $k_i=1.$ A graph $G$ is bipartite if there is a bipartition of the
set of vertices such that the edges connect only vertices in
different parts.  
 The eigenvalues and the corresponding eigenvectors of a bipartite graph has a special structure as mentioned below. 
\begin{prop}\label{p3}\cite{bro}
If $\theta$ is an eigenvalue of a bipartite graph $G$ with multiplicity $k$, then $-\theta$ is also an eigenvalue of $G$ having the same multiplicity. If
$\left[\begin{smallmatrix}
    u \\ v
\end{smallmatrix}\right]$ is an eigenvector with eigenvalue $\theta$, then $\left[\begin{smallmatrix}
    u \\ -v
\end{smallmatrix}\right]$ is an eigenvector with eigenvalue $-\theta$.
\end{prop}
 In the above Proposition \ref{p3}, the vectors $u$ and $v$ correspond to the vertices in the two partite sets of $G$.
If two vertices $a$ and $b$ are adjacent then we write $a\sim b.$ An eigenvector $v$ can be realized as a function on the vertex set $V(G)$ where $v(a)$ denotes the $a$-th component of $v.$ Then $v$ is an eigenvector of $G$ with eigenvalue $\theta$ if and only if \begin{equation}\label{eqn1}
 \theta\cdot v(a)=\sum_{b\sim a}v(b)~~\text{for all $a\in V(G)$},\end{equation}
 where the summation is taken over all vertices $b\in V(G)$ that are adjacent to $a.$  Later \eqref{eqn1} shall be used in determining the eigenvectors of $T_{l,m}.$ The spectrum of a subdivided star $SK_{1,l}$ is given in \cite{bro}, which can also be obtained using \eqref{eqn1} as
\[-\sqrt{l+1},~ (-1)^{l-1},~ 0,~ 1^{l-1},~ \sqrt{l+1}.\]
The characteristic polynomial of $SK_{1,l}$ is $x(x^2-1)^{l-1}(x^2-l-1).$  In the following section, we determine the set of linearly independent eigenvalues of $T_{l,m},$ which proves to be significant in characterizing state transfer in double subdivided stars. 
\section{Linear independence of eigenvalues}\label{s2} 
Suppose $H$ is a graph having a vertex $g.$ Then $H-g$ is the induced subgraph obtained by removing the vertex $g$ from $H.$  Recall that a double subdivided star $G:=T_{l,m}$ is considered as a pair of subdivided stars $H:=SK_{1,l}$ and $H':=SK_{1,m}$ joined by an edge to the respective coalescence vertices, say, $a$ and $b$. Using \cite[Theorem 2.2.4]{cve1}, the characteristic polynomial of $G$ can be evaluated as 
\begin{align*}
 P_G(x) & =P_{H}(x)P_{H'}(x)-P_{H-a}(x)P_{H'-b}(x)
\\ &=x(x^2-1)^{l-1}(x^2-l-1)x(x^2-1)^{m-1}(x^2-m-1)-(x^2-1)^l(x^2-1)^m
\\ &=(x^2-1)^{l+m-2}q(x),
\end{align*}
where $q(x)=x^6-(l+m+3)x^4+(lm+l+m+3)x^2-1.$ One can observe that $q(x)$ is a polynomial having only the even terms, and none of $-1,0$ and $1$ are roots of $q(x)$. Suppose $q(x)$ has the roots $\pm\theta_1$, $\pm\theta_2$, $\pm\theta_3$. Considering $Q(x)=x^3-(l+m+3)x^2+(lm+l+m+3)x-1$, we have
 $Q(x^2)=q(x),$ and hence $\theta_1^2$, $\theta_2^2$, $\theta_3^2$ are the roots of $Q(x)$. Since $Q(x)$ has no rational root, it is  irreducible  over $\mathbb{Q}$, and hence all roots of $Q(x)$ are simple. Consequently, all roots of $q(x)$ are simple as well.
 Then the spectrum of $T_{l,m}$ is	\[(-1)^{l+m-2},~ 1^{l+m-2},~\pm\theta_1,~\pm\theta_2,~\pm\theta_3.\]
Since $\theta_1^2$, $\theta_2^2$, $\theta_3^2$ are the roots of $Q(x)$, we also have the following identities.
\begin{eqnarray}
    \theta_1^2+\theta_2^2+\theta_3^2 &=& l+m+3. \label{e1}\\
\theta_1^2\theta_2^2+\theta_2^2\theta_3^2+\theta_3^2\theta_1^2&=&lm+l+m+3. \label{e2}\\
\theta_1^2\theta_2^2\theta_3^2&=&1.\label{e3}
\end{eqnarray}

 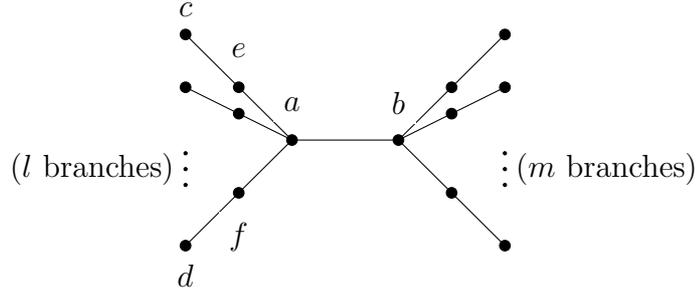
\begin{figure}
\begin{center}
\begin{tikzpicture}[scale=.7]
		\draw[black](0,0)--(2,0)-- (3,1)--(4,2);
		\draw[black](0,0)-- (-1,1)-- (-2,2);
        \draw[black](0,0)--(-1,0.5)--(-2,1);
        \draw[black] (0,0)--(-1,-1)--(-2,-2);
        \draw[black](2,0)-- (3,0.5)-- (4,1);
		\draw[black] (2,0) -- (3,-1) -- (4,-2);
	\filldraw (0,0) circle (0.1cm) node (A)  {}     node[anchor=south,fill=white,yshift=0.24cm]{$a$};
		\filldraw (2,0) circle (0.1cm) node (B) {}     node[anchor=south,fill=white,yshift=0.2cm]{$b$};
		\filldraw (-1,1) circle (0.1cm) node {}     node[anchor=south,fill=white,yshift=0.25cm]{$e$};
  \filldraw (-1,0.5) circle (0.1cm) node {}     node[anchor=south,fill=white,yshift=1cm]{};
  \filldraw (-2,1) circle (0.1cm) node {}     node[anchor=south,fill=white,yshift=0.1cm]{};
		\filldraw (-1,-1) circle (0.1cm) node  {}     node[anchor=north,fill=white,yshift=-0.25cm]{$f$};
		\filldraw (-2,2) circle (0.1cm) node  {}     node[anchor=south,fill=white,yshift=0.1cm]{$c$};
   \filldraw (-2,-0.25) circle (0.03cm) node {}     node[anchor=south,fill=white,yshift=0.4cm]{};
  \filldraw (-2,-0.55) circle (0.03cm) node [left] {($l$ branches)}     node[anchor=south,fill=white,yshift=0.6cm]{};
   \filldraw (-2,-0.85) circle (0.03cm) node {}     node[anchor=south,fill=white,yshift=0.9cm]{};
		\filldraw (-2,-2) circle (0.1cm) node  {}     node[anchor=north,fill=white,yshift=-0.10cm]{$d$};
		\filldraw (3,1) circle (0.1cm) node {}     node[anchor=south,fill=white,yshift=0.3cm]{};
  \filldraw (3,0.5) circle (0.1cm) node {}     node[anchor=south,fill=white,yshift=1cm]{};
  \filldraw (4,1) circle (0.1cm) node {}     node[anchor=south,fill=white,yshift=0.1cm]{};
   \filldraw (4,-0.25) circle (0.03cm) node {}     node[anchor=south,fill=white,yshift=0.4cm]{};
  \filldraw (4,-0.55) circle (0.03cm) node [right] {($m$ branches)}     node[anchor=south,fill=white,yshift=0.6cm]{};
   \filldraw (4,-0.85) circle (0.03cm) node {}     node[anchor=south,fill=white,yshift=0.9cm]{};
		\filldraw (3,-1) circle (0.1cm) node  {}     node[anchor=north,fill=white,yshift=-0.35cm]{};
		\filldraw (4,2) circle (0.1cm) node  {}     node[anchor=south,fill=white,yshift=0.1cm]{};
		\filldraw (4,-2) circle (0.1cm) node  {}     node[anchor=north,fill=white,yshift=-0.15cm]{};
\end{tikzpicture}	
\end{center}
 \caption{The double subdivided star $T_{l,m}.$}
 \label{fig1}
 \end{figure}
  The next result demonstrates that if the polynomial $q(x)$ is reducible, then $1,\theta_1,\theta_2$ are linearly independent over $\mathbb{Q}.$
\begin{lem}\label{lemm5}
Let $l$ and $m$ be two positive integers. Suppose $\pm\theta_1, \pm\theta_2, \pm\theta_3$ are the roots of $q(x)=x^6-(l+m+3)x^4+(lm+l+m+3)x^2-1$
in its splitting field over $\mathbb{Q}.$ If $q(x)$ is reducible, then $1,\theta_1,\theta_2$ are linearly independent over $\mathbb{Q}.$
\end{lem}
\begin{proof}
     One can observe that if the polynomial $q(x)$ is reducible for some $l$ and $m$, then it must be factored into two irreducible monic polynomials $f(x;l,m)$ and $f(-x;l,m)$ of degree three such that $q(x)=-f(x;l,m)\cdot f(-x;l,m).$ Without loss of generality, let $\theta_1,\theta_2,\theta_3$ be the distinct roots of $f(x;l,m)$, and suppose $\theta_3$ is the largest among them. Let $\alpha,\beta,\gamma \in \mathbb{Q}$ such that
\begin{equation}\label{eq27}
\alpha+\beta\theta_1+\gamma\theta_2=0.
\end{equation}
By \cite[Theorem 13.27]{dum}, the Galois group of the irreducible polynomial $f(x;l,m)$ can be realized as a transitive subgroup of $S_3$ with respect to the ordering of roots $\theta_1, \theta_2,  \theta_3$. So the Galois group of $f(x;l,m)$ must contain the alternating group $A_3=\{(1),(123),(132)\}.$ As the automorphisms in the Galois group fix $\mathbb{Q},$ the elements of $A_3$ acting on \eqref{eq27} give
\[\alpha+\beta\theta_1+\gamma\theta_2=0,~\alpha+\beta\theta_2+\gamma\theta_3=0,~
\alpha+\beta\theta_3+\gamma\theta_1=0,
\]
which is a homogeneous system of linear equations in $\alpha,~\beta$ and $\gamma.$ The coefficient matrix  can be reduced to obtain
\[\begin{bmatrix}
  1 & \theta_1 & \theta_2\\
	0 & \theta_2-\theta_1 & \theta_3-\theta_2\\
        0 & 0 & \dfrac{(\theta_1-\theta_2)^2+(\theta_3-\theta_1)(\theta_3-\theta_2)}{(\theta_1-\theta_2)}
\end{bmatrix}.\]
Since all the pivots are non zero, the rank of the coefficient matrix is $3.$ Consequently, $1,\theta_1,\theta_2$ are linearly independent over $\mathbb{Q}.$
\end{proof}
From \eqref{e1} and \eqref{e2} we find 
\begin{eqnarray} 
(\theta_1+\theta_2+\theta_3)^2&=&l+m+3+2(\theta_1\theta_2+\theta_2\theta_3+\theta_1\theta_3),\label{e35}\\(\theta_1\theta_2+\theta_2\theta_3+\theta_1\theta_3)^2&=&lm+l+m+3+2\theta_1\theta_2\theta_3(\theta_1+\theta_2+\theta_3).\label{e34}
\end{eqnarray}
If $\theta_1+\theta_2+\theta_3=0,$ then \eqref{e35} and \eqref{e34} implies that $(l-m)^2+2l+2m=3$, which is impossible as $l,m\in\mathbb{N}.$ Now Lemma \ref{lemm5} along with $\theta_1+\theta_2+\theta_3 \neq 0$ infer that if the polynomial $q(x)$ is reducible, then any proper subset of  $\{1,\theta_1,\theta_2,\theta_3\}$ is linearly independent over $\mathbb{Q}.$
\begin{thm}\label{cor2}
Let $l,m\in\mathbb{N}.$ If the polynomial $q(x)=x^6-(l+m+3)x^4+(lm+l+m+3)x^2-1$ is reducible over $\mathbb{Q},$ then the set of all positive eigenvalues of $T_{l,m}$ except one is linearly independent over $\mathbb{Q}.$  
\end{thm}
Note that if $l=m$ then $q(x)=-f(x;l,l)\cdot f(-x;l,l)$ where $f(x;l,l)=x^3-x^2-(l+1)x +1.$
The roots $\theta_1, \theta_2, \theta_3$ of $f(x;l,l)$ satisfy the following relations.
\begin{eqnarray}
\theta_1+\theta_2+\theta_3 &=& 1. \label{e10}\\
\theta_1\theta_2+\theta_2\theta_3+\theta_3\theta_1 &=&-(l+1). \label{e11}\\
\theta_1\theta_2\theta_3 &=& -1. \label{e12}
 \end{eqnarray}
 As a consequence to Theorem \ref{cor2} we have the following result.
\begin{cor}\label{cor3}
The set of all positive eigenvalues of $T_{l,l}$ except one is linearly independent over $\mathbb{Q}$ for all positive integer $l.$ 
\end{cor}
Before we proceed with the case when $q(x)$ is irreducible over $\mathbb{Q},$ consider the following result on the linear independence of $\theta_1^2,\theta_2^2$ and $\theta_3^2$ over $\mathbb{Q}.$
\begin{lem}\label{lemma3}
Let  $\theta_1^2,\theta_2^2,\theta_3^2$ be the roots of 
$Q(x)=x^3-(l+m+3)x^2+(lm+l+m+3)x-1,$ for $l,m\in\mathbb{N},$ in its splitting field over $\mathbb{Q}$. Then
$\theta_1^2,\theta_2^2,\theta_3^2$ are linearly independent over $\mathbb{Q}$.
\end{lem}
\begin{proof}
Let $\alpha,\beta,\gamma \in \mathbb{Q}$ such that
\begin{equation}\label{eq13}
\alpha\theta_1^2+\beta\theta_2^2+\gamma\theta_3^2=0.
\end{equation}
    Note that $Q(x)$ is an irreducible polynomial of degree $3$ over $\mathbb{Q}$. Since the Galois group corresponding to $Q(x)$ is transitive, it contains the alternating group $A_3.$ The elements of $A_3$ acting on \eqref{eq13} yield
\[\alpha\theta_1^2+\beta\theta_2^2+\gamma\theta_3^2=0,~
\alpha\theta_2^2+\beta\theta_3^2+\gamma\theta_1^2=0,~
\alpha\theta_3^2+\beta\theta_1^2+\gamma\theta_2^2= 0.\]
 The coefficient matrix for the corresponding homogeneous system is row equivalent to the following matrix.
 \[\begin{bmatrix}
  \theta_1^2 & \theta_2^2 & \theta_3^2\\\vspace{0.3cm}
	0 & \theta_3^2-\dfrac{\theta_2^4}{\theta_1^2} & \theta_1^2-\dfrac{\theta_2^2\theta_3^2}{\theta_1^2}\\\vspace{0.3cm}
        0 & 0 & \dfrac{3-\theta_1^6-\theta_2^6-\theta_3^6}{\theta_3^2\theta_1^2-\theta_2^4}
\end{bmatrix}.\]
Since $\theta_1^2,\theta_2^2,\theta_3^2$ are distinct real roots of $Q(x)$ satisfying $\theta_1^2\theta_2^2\theta_3^2=1,$ we have $\theta_1^6+\theta_2^6+\theta_3^6>3.$ Now $\theta_3^2-\dfrac{\theta_2^4}{\theta_1^2}=0$ infers that $\theta_2^6=1$ or $\theta_2^2=1,$ a contradiction. So all three pivots are non-zero, and therefore, the rank of the coefficient matrix is $3.$ Hence $\theta_1^2, \theta_2^2, \theta_3^2$ are linearly independent over $\mathbb{Q}.$
\end{proof}
Suppose $q(x)=x^6-(l+m+3)x^4+(lm+l+m+3)x^2-1$ is irreducible over $\mathbb{Q}$, and consider the Galois group $\mathcal{G}$ of $q(x)$ that fixes $\mathbb{Q}.$ Applying \cite[Theorem 13.27]{dum}, the Galois group $\mathcal{G}$ can be realized as a transitive subgroup of $S_6$ with respect to the ordering of roots $\theta_1, -\theta_1, \theta_2, -\theta_2, \theta_3, -\theta_3$ of $q(x).$  Since the discriminant $D$ of  $Q(x)$ satisfy $D=(\theta_2^2-\theta_1^2)^2(\theta_3^2-\theta_1^2)^2(\theta_3^2-\theta_2^2)^2\in\mathbb{Q},$  
 the discriminant of $q(x)$ evaluated as $64D^2$ is a square of an element in $\mathbb{Q}.$ Using \cite[Proposition 14.34]{dum}, the Galois group $\mathcal{G}$ is a transitive subgroup of the alternating group $A_6$. 
Now we have the following result.
\begin{thm}\label{t9}
Let $l$ and $m$ be two positive integers. Suppose
 $\pm\theta_1, \pm\theta_2, \pm\theta_3$ are the roots of $q(x)=x^6-(l+m+3)x^4+(lm+l+m+3)x^2-1$
in its splitting field over $\mathbb{Q}.$ If $q(x)$ is irreducible then $1,\theta_1,\theta_2,\theta_3$ are linearly independent over $\mathbb{Q}.$
 \end{thm}
\begin{proof}  
Let $\alpha,\beta,\gamma,\delta\in \mathbb{Q}$ such that
\begin{equation}\label{e39} \alpha+\beta\theta_1+\gamma\theta_2+\delta\theta_3=0.
\end{equation}

Since the Galois group $\mathcal{G}$ of $q(x)$ is  a transitive subgroup of $A_6$, there is an automorphism $\sigma\in \mathcal{G}$ such that $\sigma(\theta_1)=-\theta_1.$ Applying $\sigma$ on both sides of \eqref{e39} and adding the resulting equation to it yields
\begin{equation}\label{e41} 2\alpha+\gamma(\theta_2+\sigma(\theta_2))+\delta(\theta_3+\sigma(\theta_3))=0.
\end{equation}
Each automorphism in $\mathcal{G}$ that maps $\theta_i$ to $\theta_j$ must map $-\theta_i$ to $-\theta_j.$
% Thus $\mathcal{G}$ does not contain any automorphisms having exactly one cycle of length three or five.
Since $\mathcal{G}$ is a subgroup of $A_6$, the only possibility for $\sigma$ remains $(12)(34)$, $(12)(56)$, $(12)(3546)$ or $(12)(3645).$

If $\sigma=(12)(34)$, then \eqref{e41} becomes $\alpha+\delta\theta_3=0.$ Since $\theta_3\notin\mathbb{Q}$, we obtain $\alpha=\delta=0.$ Thus \eqref{e39} reduces to  $\beta\theta_1+\gamma\theta_2=0,$ which further gives $\beta =\gamma=0$ as Lemma \ref{lemma3} holds. Therefore $1,\theta_1,\theta_2,\theta_3$ are linearly independent over $\mathbb{Q}.$ Using similar argument for $\sigma=(12)(56)$, we arrive at the same conclusion. If $\sigma=(12)(3546)$ then we find $\sigma^{-1}=(12)(3645)$. Now \eqref{e41} becomes
\begin{equation}\label{e42}
 2\alpha+(\gamma-\delta)\theta_2+(\gamma+\delta)\theta_3=0.  
\end{equation} 
Applying $\sigma^2=(34)(56)$ on both sides of \eqref{e42} gives $\alpha=0.$ Finally by Lemma \ref{lemma3}, we find $\alpha=\beta=\gamma=\delta=0.$ Hence $1,\theta_1,\theta_2,\theta_3$ are linearly independent over $\mathbb{Q}.$
\end{proof}
Consequently from Theorem \ref{t9} we obsereve that if $q(x)$ is irreducible over $\mathbb{Q}$ then the set of all positive eigenvalues of $T_{l,m}$ is linearly independent over $\mathbb{Q}.$
\begin{cor}\label{cor4}
If $q(x)=x^6-(l+m+3)x^4+(lm+l+m+3)x^2-1, \text{ for } l,m\in\mathbb{N},$ is irreducible over $\mathbb{Q},$ then the set of all positive (negative) eigenvalues of $T_{l,m}$ is linearly independent over $\mathbb{Q}.$ \end{cor}

\section{State transfer on $T_{l,m}$}\label{s3}
In quest of the existence of PST (or PGST) in $T_{l,m}$ from a vertex $a,$ we analyze the eigenvalues in the support $\sigma_a$ and the corresponding orthogonal projections. In this regard, we determine the eigenvectors of $T_{l,m}$ corresponding to each of its eigenvalues. Recall that the eigenvectors are real-valued functions on the vertex set of $T_{l,m}.$ In case $l>1$ (or $m>1$), an eigenvector corresponding to $1$ can be obtained by assigning the value $1$ to a pair of adjacent vertices of degree $1$ and $2$ in a branch, $-1$ to another such pair in a branch adjacent to the previous one, and the remaining vertices are assigned $0.$ Considering other such adjacent branches, we obtain $l+m-2$ linearly independent eigenvectors corresponding to $1$. Similarly, a set of $l+m-2$ linearly independent eigenvectors for the eigenvalue $-1$ can be obtained using \eqref{eqn1}. Suppose $E_{-1}$  and $E_1$ are idempotents corresponding to $-1$ and $1$, respectively. For the vertices $a$ and $b$ in $T_{l,m}$ (see Figure \ref{fig1}), note that $-1$ and $1$ are not in $\sigma_a$ as well as  $\sigma_b$ since we have
$E_{-1}\e_a=E_{-1}\e_b=0$ and $E_{1}\e_a=E_{1}\e_b=0$. 
However, we observed that  
$E_{-1}\e_c=E_{-1}\e_d\neq0~\text{and}~
E_{1}\e_c=E_{1}\e_d\neq0,$ and hence both $-1$ and $1$ belong to $\sigma_c$ as well as $\sigma_d.$ 

The eigenvectors for the remaining eigenvalues are obtained as follows. Let $P$ be the permutation matrix corresponding to an automorphism of $T_{l,m}$. Note that $P$ commutes with the adjacency matrix $A$ of $T_{l,m}$. Let $v$ be an eigenvector of
$T_{l,m}$ satisfying $Av = \theta v$ with $\theta \neq -1, 1.$ Now $APv = PAv = \theta Pv$ implies that $Pv$ is also an eigenvector corresponding to $\theta$. Suppose the entries of $v$ are  $z_1,z_2,x_j,y_j,u_k,w_k$, where $j=1,2,\ldots,l$ and $k=1,2,\ldots,m$  as mentioned in Figure \ref{fig2}.
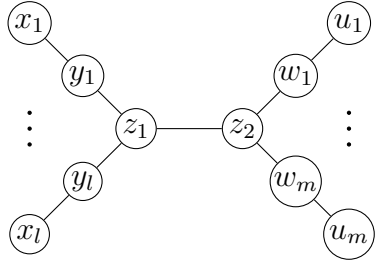
\begin{figure}
\begin{center}
\begin{tikzpicture}[scale=0.7]
   \node[circle,draw,inner sep=1pt] (A) at (0,0){$z_1$};
  \node[circle,draw,inner sep=1pt] (B) at (2,0){$z_2$};
  \draw (A) -- (B);
  \node[circle,draw,inner sep=1pt] (C) at (-1,1){$y_1$};
  \node[circle,draw,inner sep=1pt] (D) at (-1,-1){$y_l$};
  \draw (A) -- (C);
  \draw (A) -- (D);
  \node[circle,draw,inner sep=1pt] (E) at (-2,2){$x_1$};
  \node[circle,draw,inner sep=1pt] (F) at (-2,-2){$x_l$};
  \draw (C) -- (E);
  \draw (D) -- (F);
  \node[circle,draw,inner sep=0.5pt] (G) at (3,1){$w_1$};
  \node[circle,draw,inner sep=0.5pt] (H) at (3,-1){$w_m$};
  \draw (B) -- (G);
  \draw (B) -- (H);
  \node[circle,draw,inner sep=1pt] (I) at (4,2){$u_1$};
  \node[circle,draw,inner sep=1pt] (J) at (4,-2){$u_m$};
  \draw (G) -- (I);
  \draw (H) -- (J);
   \filldraw (-2,0.3) circle (0.03cm) node {}     node[anchor=south,fill=white,yshift=0.4cm]{};
  \filldraw (-2,0) circle (0.03cm) node {} node[anchor=south,fill=white,yshift=0.6cm]{};
   \filldraw (-2,-0.3) circle (0.03cm) node {}   node[anchor=south,fill=white,yshift=0.8cm]{};
   \filldraw (4,0.3) circle (0.03cm) node {}     node[anchor=south,fill=white,yshift=0.4cm]{};
  \filldraw (4,0) circle (0.03cm) node {} node[anchor=south,fill=white,yshift=0.6cm]{};
   \filldraw (4,-0.3) circle (0.03cm) node {}   node[anchor=south,fill=white,yshift=0.8cm]{};
\end{tikzpicture}
\end{center}
\caption{An eigenvector corresponding to $\theta$.}
\label{fig2}
\end{figure}
In particular, suppose $P$ is an automorphism of $T_{l,m}$ which switches vertices assigned with entries $x_1$ and $x_2$, $y_1$ and $y_2$, and fixing all other vertices. 
Since all eigenvalues except $-1$ and $1$ are simple, the eigenvectors $v$ and $Pv$ are parallel. As a result $Pv=\alpha v$ for some scalar $\alpha$, which further gives $\alpha z_1=z_1$. If $z_1=0$ then \eqref{eqn1} infers that $\theta x_1=y_1$ and $\theta y_1=x_1$, which is absurd as $\theta\neq\pm1.$ Hence $\alpha=1,$ and we have $x_1=x_2$ and $y_1=y_2$. We  therefore conclude $x_1=x_j$, $y_1=y_j$, $u_1=u_k$ and $w_1=w_k$ for all $j$ and $k$.

In case $l=m,$ consider an automorphism $P'$ which switches vertices assigned with entries $z_1$ and $z_2$. 
Since $v$ and $P'v$ are parallel, $P'v=\beta v$ for some scalar $\beta$. This gives
$z_2=\beta z_1$, $z_1=\beta z_2$, $u_1=\beta x_1$ and $w_1=\beta y_1$. Consequently, we have $\beta=\pm 1$.
 The eigenvector $v$ corresponding to the case $\beta=1$ is given in Figure \ref{fig3}. For $\beta=-1,$ the eigenvector can be obtained similarly.
\begin{figure}
\begin{center}
\begin{tikzpicture}[scale=0.7]
   \node[circle,draw,inner sep=1pt] (A) at (0,0){$z_1$};
  \node[circle,draw,inner sep=1pt] (B) at (2,0){$z_1$};
  \draw (A) -- (B);
  \node[circle,draw,inner sep=1pt] (C) at (-1,1){$y_1$};
  \node[circle,draw,inner sep=1pt] (D) at (-1,-1){$y_1$};
  \draw (A) -- (C);
  \draw (A) -- (D);
  \node[circle,draw,inner sep=1pt] (E) at (-2,2){$x_1$};
  \node[circle,draw,inner sep=1pt] (F) at (-2,-2){$x_1$};
  \draw (C) -- (E);
  \draw (D) -- (F);
  \node[circle,draw,inner sep=1pt] (G) at (3,1){$y_1$};
  \node[circle,draw,inner sep=1pt] (H) at (3,-1){$y_1$};
  \draw (B) -- (G);
  \draw (B) -- (H);
  \node[circle,draw,inner sep=1pt] (I) at (4,2){$x_1$};
  \node[circle,draw,inner sep=1pt] (J) at (4,-2){$x_1$};
  \draw (G) -- (I);
  \draw (H) -- (J);
  \filldraw (-2,0.3) circle (0.03cm) node {}     node[anchor=south,fill=white,yshift=0.4cm]{};
  \filldraw (-2,0) circle (0.03cm) node [left] {($l$ branches)}  node[anchor=south,fill=white,yshift=0.6cm]{};
   \filldraw (-2,-0.3) circle (0.03cm) node {}   node[anchor=south,fill=white,yshift=0.8cm]{};
   \filldraw (4,0.3) circle (0.03cm) node {}     node[anchor=south,fill=white,yshift=0.4cm]{};
  \filldraw (4,0) circle (0.03cm) node [right] {($l$ branches)} node[anchor=south,fill=white,yshift=0.6cm]{};
   \filldraw (4,-0.3) circle (0.03cm) node {}   node[anchor=south,fill=white,yshift=0.8cm]{};
\end{tikzpicture}
\end{center}
\caption{The eigenvector $v$ for $\beta=1.$}
\label{fig3}
\end{figure}
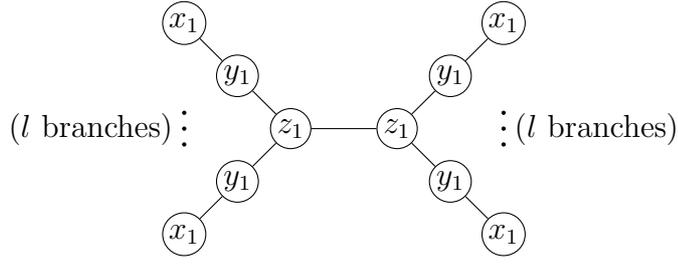
The following result shows that the double subdivided star $T_{l,m}$ does not exhibit PST for any positive integers $l~\text{and}~m.$ 
\begin{thm}\label{t2}
    There is no perfect state transfer in the double subdivided star $T_{l,m}$ for any positive integers $l$ and $m.$
\end{thm}
\begin{proof}
Recall that the spectrum of $T_{l,m}$ is 
 \[(-1)^{l+m-2},~1^{l+m-2},~\pm\theta_1,~\pm\theta_2,~\pm\theta_3.\]
It is evident that the nonzero eigenvalues $\pm\theta_1, \pm\theta_2, \pm\theta_3$ are in the eigenvalue support of the vertices $a,c~\text{and}~e$ as in Figure \ref{fig1}. Suppose $T_{l,m}$ is periodic at vertex $a$. Then the ratio condition in Theorem \ref{p6} gives
 \[\dfrac{\theta_1-(-\theta_1)}{\theta_2-(-\theta_2)}=\dfrac{\theta_1}{\theta_2}\in\mathbb{Q},\]
  which is a contradiction to Lemma \ref{lemm5} or Theorem \ref{t9} depending on whether the polynomial $q(x)=x^6-(l+m+3)x^4+(lm+l+m+3)x^2-1$ is reducible or irreducible over $\mathbb{Q}$. Therefore $T_{l,m}$ is not periodic at $a,$ and hence there is no PST from the vertex $a$. Similarly, there is no PST from the vertices $c$ and $e$ as well. Hence the result follows.
 \end{proof}
Whenever $l=m=1,$ the double subdivided star $T_{l,m}$ becomes the path $P_6$. The fact that $P_6$ does not exhibit PST was previously observed in \cite{god1} as well. Next we investigate the existence of PGST in $T_{l,m}.$
\subsection{Pretty good state transfer}
 The graph $T_{l,m}$ becomes the path $P_6$ for $l=m=1$. The existence of PGST in $P_6$ is mentioned in \cite{god4}, where we find that there is PGST from all vertices of $P_6$. Now we consider the remaining cases. It follows from Lemma \ref{l4} that if there is PGST between a pair of vertices then they have the same degree. It is well known that if there is PGST between two vertices then each automorphism fixing one must fix the other. Therefore there is no pretty good state transfer in $T_{l,m}$ whenever $l$ and $m$ are distinct and $l\neq 2\neq m.$
Next we investigate the existence of PGST in $T_{l,l}.$ 
\begin{thm}\label{th1}
 There exists pretty good state transfer between the coalescence vertices in $T_{l,l}$ with respect to a sequence in $(4\mathbb{Z}-1)\dfrac{\pi}{2}$ for all natural number $l.$  
\end{thm}
\begin{proof}
The eigenvalue support of the coalescence vertices $a$ and $b$ in $T_{l,l}$ are \[\sigma_a=\{\pm\theta_1,\pm\theta_2,\pm\theta_3\}=\sigma_b.\] Recall that the eigenvalues $\pm\theta_1,\pm\theta_2,\pm\theta_3$ are simple. Suppose $v_1,v_2,v_3$ are the eigenvectors corresponding to  $\theta_1,\theta_2,\theta_3$, respectively. Using Proposition $\ref{p3}$, we determine the eigenvectors corresponding to $-\theta_1,-\theta_2,-\theta_3$ as well. Therefore
\begin{eqnarray}
\e_a^TU(t)\e_b&=&\sum_{\theta\in \sigma_a}\exp{(it\theta)}\e_a^TE_{\theta}\e_b\nonumber \\&=& \sum_{j=1}^{3}  \left[\exp{(it\theta_j)}\dfrac{v_j(a)v_j(b)}{||v_j||^2}-\exp{(-it\theta_j)}\dfrac{v_j(a)v_j(b)}{||v_j||^2}\right] 
\end{eqnarray}
We already showed that $v_j(a)=v_j(b)\neq 0.$ Without loss of generality, let $v_j(a)=1$ for $j=1,2,3.$
Thus the above equation yields
\begin{eqnarray}\label{eqn17}
\e_a^TU(t)\e_b&=&  \sum_{j=1}^{3} \left[\dfrac{\exp{(it\theta_j)}-\exp{(-it\theta_j)}}{||v_j||^2}\right]
\end{eqnarray}
 By Lemma \ref{lemm5}, the algebraic numbers $1,\theta_1,\theta_2$ are  linearly independent over $\mathbb{Q}.$ Let $\epsilon>0,$ and consider $\alpha_j=\dfrac{1+\theta_j}{4}$ in Theorem \ref{t1}. Then there exist $q,p_1,p_2\in\mathbb{Z}$ such that
\begin{equation}\label{e17}
\left|(4q-1)\dfrac{\pi}{2}\theta_j-\left(2\pi p_j
+\dfrac{\pi}{2}\right)\right|<2\pi\epsilon ~~\text{for}~~ j=1,2.
\end{equation}
Since $\theta_1+\theta_2+\theta_3= 1$ as in \eqref{e10}, this further yields  \begin{equation}\label{e18}\left|(4q-1)\dfrac{\pi}{2}\theta_3-\left(\pi(2(q-p_1-p_2)-1) -\dfrac{\pi}{2}\right)\right|<4\pi\epsilon.
\end{equation}
We obtain a sequence $\tau_k\in(4\mathbb{Z}-1)\dfrac{\pi}{2}$ from \eqref{e17} and \eqref{e18} such that $\displaystyle\lim_{k\to\infty} \exp{(i\tau_k\theta_j)}=i$ for all $j.$ Thus \eqref{eqn17} together with the fact that $U(0)=I$ gives 
\begin{equation*}\label{e19}\lim_{k\to \infty}\e_a^TU(\tau_k)\e_b=2i\left[\dfrac{1}{||v_1||^2}+\dfrac{1}{||v_2||^2}+\dfrac{1}{||v_3||^2}\right]=i.
\end{equation*}
This completes the proof.
\end{proof}
In case of $P_6,$ the support of each vertex contains all its eigenvalues. The proof of Theorem \ref{th1} can be devised to determine that $P_6$ exhibits PGST between the pair of vertices $n$ and $7-n$ for all $n=1,2,3.$

Now we investigate the existence of PGST in $T_{2,m}$ for some positive integer $m.$ Note that there is no PGST between the coalescence vertices of $T_{2,m}$ whenever $m\neq 2.$ Therefore, if $T_{2,m}$ admits PGST then it occurs between the pair of vertices $c, d$ or the pair of vertices $e, f$ as in Figure \ref{fig4}. Here we find that 
\[\sigma_c=\sigma_d=\sigma_e=\sigma_f=\{\pm 1,\pm\theta_1,\pm\theta_2,\pm\theta_3\},\]
which is the set of all eigenvalues of $T_{2,m}$. Let $v_1, v_2, v_3$ be the  eigenvectors corresponding to the simple eigenvalues $\theta_1, \theta_2, \theta_3$, respectively. Then Proposition $\ref{p3}$ gives the eigenvectors corresponding to $-\theta_1, -\theta_2, -\theta_3$ as well. Recall that $E_{-1}$ and $E_1$ are the idempotents corresponding to eigenvalues $-1$ and $1,$ respectively. Using Gram-Schmidt  procedure on the set of linearly independent eigenvectors corresponding to $-1$ and $1$, we evaluate
\[ \e_c^TE_{-1}\e_ d=\e_c^TE_{1}\e_ d=-\dfrac{1}{4}=\e_e^TE_{-1}\e_ f=\e_e^TE_{1}\e_ f.\]
Now we have
\begin{eqnarray*}
\e_c^TU(t)\e_d 
&=&\sum_{j=1}^{3}\left[\exp{(it\theta_j)}\dfrac{v_j(c)v_j(d)}{||v_j||^2}+\exp{(-it\theta_j)}\dfrac{v_j(c)v_j(d)}{||v_j||^2}\right]\\
&&+\exp{(it)}\left(-\dfrac{1}{4}\right)+\exp{(-it)}\left(-\dfrac{1}{4}\right).
\end{eqnarray*}

We already obtained $v_j(c)=v_j(d)\neq 0.$ Without loss of generality, let $v_j(c)=1$ for $j=1,2,3.$ Therefore
\begin{eqnarray}\label{eq26}
\e_c^TU(t)\e_d =\sum_{j=1}^{3}\left[\dfrac{\exp{(it\theta_j)}+\exp{(-it\theta_j)}}{||v_j||^2}\right]+\dfrac{1
}{4}\left[\exp{(i(t+\pi))}+\exp{(-i(t-\pi))}\right].
\end{eqnarray}
Similarly, we use a different set of eigenvectors $v_1, v_2, v_3$ satisfying $v_j(e)=v_j(f)=1$ for all $j$ to obtain 
\begin{eqnarray}\label{eqn21}
\e_e^TU(t)\e_f =\sum_{j=1}^{3}\left[\dfrac{\exp{(it\theta_j)}+\exp{(-it\theta_j)}}{||v_j||^2}\right]+\dfrac{1
}{4}\left[\exp{(i(t+\pi))}+\exp{(-i(t-\pi))}\right].
\end{eqnarray}
In case $l=2$ and $m$ is any positive integer, the polynomial $q(x)$ for the graph $T_{2,m}$ becomes $q(x)=x^6-(m+5)x^4+(3m+5)x^2-1.$ Next we classify the existence of PGST in $T_{2,m}.$  
\begin{figure}
\begin{center}
\begin{tikzpicture}[scale=.7]
		\draw[black](0,0)--(2,0)-- (3,1)--(4,2);
		\draw[black](0,0)-- (-1,1)-- (-2,2);
        \draw[black] (0,0)--(-1,-1)--(-2,-2);
        \draw[black](2,0)-- (3,0.5)-- (4,1);
		\draw[black] (2,0) -- (3,-1) -- (4,-2);
	\filldraw (0,0) circle (0.1cm) node (A)  {}     node[anchor=south,fill=white,yshift=0.24cm]{};
		\filldraw (2,0) circle (0.1cm) node (B) {}     node[anchor=south,fill=white,yshift=0.2cm]{};
		\filldraw (-1,1) circle (0.1cm) node {}     node[anchor=south,fill=white,yshift=0.25cm]{$e$};
		\filldraw (-1,-1) circle (0.1cm) node  {}     node[anchor=north,fill=white,yshift=-0.25cm]{$f$};
		\filldraw (-2,2) circle (0.1cm) node  {}     node[anchor=south,fill=white,yshift=0.1cm]{$c$};
		\filldraw (-2,-2) circle (0.1cm) node  {}     node[anchor=north,fill=white,yshift=-0.10cm]{$d$};
		\filldraw (3,1) circle (0.1cm) node {}     node[anchor=south,fill=white,yshift=0.3cm]{};
  \filldraw (3,0.5) circle (0.1cm) node {}     node[anchor=south,fill=white,yshift=1cm]{};
  \filldraw (4,1) circle (0.1cm) node {}     node[anchor=south,fill=white,yshift=0.1cm]{};
   \filldraw (4,-0.25) circle (0.03cm) node {}     node[anchor=south,fill=white,yshift=0.4cm]{};
  \filldraw (4,-0.55) circle (0.03cm) node [right] {($m$ branches)}     node[anchor=south,fill=white,yshift=0.6cm]{};
   \filldraw (4,-0.85) circle (0.03cm) node {}     node[anchor=south,fill=white,yshift=0.9cm]{};
		\filldraw (3,-1) circle (0.1cm) node  {}     node[anchor=north,fill=white,yshift=-0.35cm]{};
		\filldraw (4,2) circle (0.1cm) node  {}     node[anchor=south,fill=white,yshift=0.1cm]{};
		\filldraw (4,-2) circle (0.1cm) node  {}     node[anchor=north,fill=white,yshift=-0.15cm]{};
\end{tikzpicture}	
\end{center}
 \caption{The double subdivided star $T_{2,m}.$~~~~~~~~~~~~~~}
 \label{fig4}
 \end{figure}
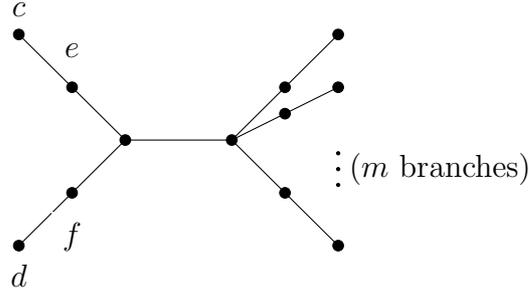
\begin{thm}\label{t8}
Let $m$ be a positive integer. Suppose $\pm\theta_1, \pm\theta_2, \pm\theta_3$ are the roots of the polynomial $q(x)=x^6-(m+5)x^4+(3m+5)x^2-1$ 
in its splitting field over $\mathbb{Q}.$ Then the following holds in $T_{2,m}$.
\begin{enumerate}
\item If $q(x)$ is irreducible over $\mathbb{Q}$, then there is pretty good state transfer with respect to a sequence in $(2\mathbb{Z}+1)\pi$ between both pair of vertices $c,d$ and $e,f$. 
 \item Let $q(x)$ is reducible over $\mathbb{Q}.$ If $\theta_1+\theta_2+\theta_3$ is an even integer, then there is pretty good state transfer with respect to a sequence in $(2\mathbb{Z}+1)\pi$ between both pair of vertices $c,d$ and $e,f$. Moreover, if $\theta_1+\theta_2+\theta_3$ is an odd integer then there is no pretty good state transfer from $c,d,e$ and $f.$
 \end{enumerate}
\end{thm}
\begin{proof} 
Suppose $q(x)$ is irreducible. Then by Theorem \ref{t9}, the algebraic numbers $1, \theta_1, \theta_2, \theta_3$ are linearly independent over $\mathbb{Q}.$ Let $\epsilon>0$ and consider $\alpha_j=-\dfrac{\theta_j}{2},$ for $j=1,2,3.$ By Theorem \ref{t1}, there exist $q,p_1,p_2\in\mathbb{Z}$ such that 
\begin{equation}
\left|(2q+1)\pi\theta_j-2\pi p_j\right|<2\pi\epsilon~\text{ for}~j=1,2,3.
\end{equation}
This along with \eqref{eq26} gives a sequence $\tau_k\in(2\mathbb{Z}+1)\pi$ such that 
\[\displaystyle\lim_{k\to \infty}\e_c^TU(\tau_k)\e_d=2\left[\dfrac{1}{||v_1||^2}
+\dfrac{1}{||v_2||^2}+\dfrac{1}{||v_3||^2}+\dfrac{1}{4}\right]=1,\]
since $U(0)=I.$ Therefore, PGST occurs between the pair of vertices $c$ and $d$ with respect to the sequence $\tau_k\in(2\mathbb{Z}+1)\pi.$ Similarly using \eqref{eqn21}, we find that $T_{2,m}$ exhibits PGST between the pair of vertices $e$ and $f$ with respect to the same sequence $\tau_k$ as well.

Suppose $q(x)$ is reducible and $\theta_1+\theta_2+\theta_3=2n$ for some $n\in\mathbb{Z}$. Using Lemma \ref{lemm5}, the algebraic numbers $1,\theta_1,\theta_2$ are linearly independent over $\mathbb{Q}$. Let $\epsilon>0$ and consider $\alpha_j=-\dfrac{\theta_j}{2}$ whenever $j=1,2.$ By Theorem \ref{t1}, there exist $q,p_1,p_2\in\mathbb{Z}$ such that 
\begin{equation}\label{e36}
	\left|(2q+1)\pi\theta_j-2\pi p_j\right|<2\pi\epsilon,~\text{ for}~j=1,2.
\end{equation}
This further yields
\begin{equation}\label{e38}
	 \left|(2q+1)\pi\theta_3-2\pi(2qn+n-p_1-p_2)\right|<4\pi\epsilon.
\end{equation}
Using \eqref{e36} and \eqref{e38}, we obtain a sequence  $\tau_k\in(2\mathbb{Z}+1)\pi$ such that $\displaystyle\lim_{k\to\infty} \exp{(i\tau_k\theta_j)}=1,$ for $j=1,2,3.$
Using \eqref{eq26}, we have 
\[\displaystyle\lim_{k\to\infty}e_c^TU(\tau_k)e_d=2\left[\dfrac{1}{||v_1||^2}+\dfrac{1}{||v_2||^2}+\dfrac{1}{||v_3||^2}+\dfrac{1}{4}\right]=1.\]
 Hence $T_{2,m}$ exhibits PGST between the pair of vertices $c$ and $d$ with respect to the sequence $\tau_k\in(2\mathbb{Z}+1)\pi$. Similarly using \eqref{eqn21} we find that $T_{2,m}$ exhibits PGST between the pair of vertices $e$ and $f$ with respect to the same sequence $\tau_k$ as well.

Finally, consider the case that $q(x)$ is reducible and $\theta_1+\theta_2+\theta_3=2n+1$ for some $n\in\mathbb{Z}.$
In the proof of the main result in \cite{god4}, one can observe that if there is PGST in a bipartite graph between a pair of vertices $a$ and $b$ with $\displaystyle\lim_{k\to\infty}U(\tau_k)\e_a=\gamma\e_b,~\text{for some}~ \tau_k\in\mathbb{R} \text{ and }\gamma\in \mathbb{C}$ then $\gamma=\pm 1$ whenever $a$ and $b$ are in the same partite set, otherwise $\gamma=\pm i$. Since $U(0)=I$, we conclude from \eqref{eq26} that if there is PGST between $c$ and $d$, then we have a sequence $\tau_k\in\mathbb{R}$ such that for all $j=1,2,3,$\[\displaystyle\lim_{k\to\infty}\exp{(i(\tau_k+\pi))}=\displaystyle\lim_{k\to\infty}\exp{(i\tau_k\theta_j)}=\pm 1.\] 
In case $\displaystyle\lim_{k\to\infty}\exp{(i(\tau_k+\pi))}=1,$ it follows that $\tau_k\in(2\mathbb{Z}+1)\pi.$ Since $\theta_1+\theta_2+\theta_3=2n+1,$ we have a contradiction that
$-1=\displaystyle\lim_{k\to\infty}\exp{\left[i\tau_k\left(\theta_1+\theta_2+\theta_3\right)\right]}=1,$
where the equality on the right is obtained by using the property of exponentials. When $\displaystyle\lim_{k\to\infty}\exp{(i(\tau_k+\pi))}=-1,$ we have $\tau_k\in 2\pi\mathbb{Z},$ and again it leads to a contradiction. Hence there is no PGST between the vertices $c$ and $d.$ Using \eqref{eqn21} and a similar argument, we conclude that there is no PGST between the vertices $e$ and $f$ as well.  
\end{proof}
 Considering $\alpha_j=\frac{1}{2}$ in the proof of Theorem \ref{t8} where $q(x)$ is irreducible, one can deduce that there is PGST in $T_{2,m}$ with respect to a sequence in $2\pi\mathbb{Z}$ between the same pair of vertices.
Combining Theorem \ref{th1} and Theorem \ref{t8}, one obtains a complete characterization of double subdivided stars $T_{l,m}$ exhibiting PGST. The method presented in this paper may be devised to classify family of graphs exhibiting other such quantum transportation phenomena, such as quantum fractional revival, pretty good fractional revival, etc.

% \section{Conclusions}
% We have investigated the existence of PST and PGST on double subdivided star $T_{l,m}$ by analyzing its linearly independent eigenvalues. Using the Galois group of the characteristic polynomial, a complete characterization on the linear independence of the eigenvalues of $T_{l,m}$ has been presented as well. Then we have showed that there is no PST in $T_{l,m}$ for any positive integer $l$ and $m$. However, there exists double subdivided stars $T_{l,m}$ exhibiting PGST. The method presented in this paper may be devised to classify family of graphs having  quantum state transfer properties, such as, PST, PGST, quantum fractional revival, pretty good fractional revival, etc.

\section*{Acknowledgements}
The authors are indebted to the reviewers for the valuable comments and generous suggestions to improve the manuscript.
S. Mohapatra is supported by Department of Science and Technology (INSPIRE: IF210209). H. Pal is funded by Science and Engineering Research Board (Project: SRG/2021/000522). 

%\newpage

%%%%%%%%%%%%%[The Bibliography]%%%%%%%%%%%%

\bibliographystyle{abbrv}
\bibliography{references}

\begin{thebibliography}{10}

\bibitem{ack1}
E.~Ackelsberg, Z.~Brehm, A.~Chan, J.~Mundinger, and C.~Tamon.
\newblock Quantum state transfer in coronas.
\newblock {\em Electron. J. Combin.}, 24(2):Paper No. 2.24, 26, 2017.

\bibitem{ange1}
R.~J. Angeles-Canul, R.~M. Norton, M.~C. Opperman, C.~C. Paribello, M.~C.
  Russell, and C.~Tamon.
\newblock Perfect state transfer, integral circulants, and join of graphs.
\newblock {\em Quantum Inf. Comput.}, 10(3-4):325--342, 2010.

\bibitem{apo}
T.~M. Apostol.
\newblock {\em Modular functions and {D}irichlet series in number theory},
  volume~41 of {\em Graduate Texts in Mathematics}.
\newblock Springer-Verlag, New York, second edition, 1990.

\bibitem{mil4}
M.~Ba\v{s}i\'{c}.
\newblock Characterization of quantum circulant networks having perfect state
  transfer.
\newblock {\em Quantum Inf. Process.}, 12(1):345--364, 2013.

\bibitem{pal9}
B.~Bhattacharjya, H.~Monterde, and H.~Pal.
\newblock Quantum walks on blow-up graphs.
\newblock {\em arXiv preprint arXiv:2308.13887}, 2023.

\bibitem{bose}
S.~Bose.
\newblock Quantum communication through an unmodulated spin chain.
\newblock {\em Physical review letters}, 91:207901, 2003.

\bibitem{bro}
A.~E. Brouwer and W.~H. Haemers.
\newblock {\em Spectra of graphs}.
\newblock Universitext. Springer, New York, 2012.

\bibitem{che}
W.-C. Cheung and C.~Godsil.
\newblock Perfect state transfer in cubelike graphs.
\newblock {\em Linear Algebra Appl.}, 435(10):2468--2474, 2011.

\bibitem{chr1}
M.~Christandl, N.~Datta, A.~Ekert, and A.~J. Landahl.
\newblock Perfect state transfer in quantum spin networks.
\newblock {\em Physical review letters}, 92:187902, 2004.

\bibitem{cou6}
G.~Coutinho, P.~F. Baptista, C.~Godsil, T.~J. Spier, and R.~Werner.
\newblock Irrational quantum walks.
\newblock {\em SIAM J. Appl. Algebra Geom.}, 7(3):567--584, 2023.

\bibitem{cou5}
G.~Coutinho, C.~Godsil, E.~Juliano, and C.~M. van Bommel.
\newblock Quantum walks do not like bridges.
\newblock {\em Linear Algebra Appl.}, 652:155--172, 2022.

\bibitem{cou3}
G.~Coutinho, K.~Guo, and C.~M. van Bommel.
\newblock Pretty good state transfer between internal nodes of paths.
\newblock {\em Quantum Inf. Comput.}, 17(9-10):825--830, 2017.

\bibitem{cou7}
G.~Coutinho, E.~Juliano, and T.~J. Spier.
\newblock No perfect state transfer in trees with more than 3 vertices.
\newblock {\em arXiv preprint arXiv:2305.10199}, 2023.

\bibitem{cve1}
D.~Cvetkovi\'{c}, P.~Rowlinson, and S.~Simi\'{c}.
\newblock {\em An introduction to the theory of graph spectra}, volume~75 of
  {\em London Mathematical Society Student Texts}.
\newblock Cambridge University Press, Cambridge, 2010.

\bibitem{dum}
D.~S. Dummit and R.~M. Foote.
\newblock {\em Abstract algebra}.
\newblock John Wiley \& Sons, Inc., Hoboken, NJ, third edition, 2004.

\bibitem{eis}
O.~Eisenberg, M.~Kempton, and G.~Lippner.
\newblock Pretty good quantum state transfer in asymmetric graphs via
  potential.
\newblock {\em Discrete Math.}, 342(10):2821--2833, 2019.

\bibitem{fan}
X.~Fan and C.~Godsil.
\newblock Pretty good state transfer on double stars.
\newblock {\em Linear Algebra Appl.}, 438(5):2346--2358, 2013.

\bibitem{farhi}
E.~Farhi and S.~Gutmann.
\newblock Quantum computation and decision trees.
\newblock {\em Phys. Rev. A (3)}, 58(2):915--928, 1998.

\bibitem{god3}
C.~Godsil.
\newblock Periodic graphs.
\newblock {\em Electron. J. Combin.}, 18(1):Paper 23, 15, 2011.

\bibitem{god1}
C.~Godsil.
\newblock State transfer on graphs.
\newblock {\em Discrete Math.}, 312(1):129--147, 2012.

\bibitem{god2}
C.~Godsil.
\newblock When can perfect state transfer occur?
\newblock {\em Electron. J. Linear Algebra}, 23:877--890, 2012.

\bibitem{god4}
C.~Godsil, S.~Kirkland, S.~Severini, and J.~Smith.
\newblock Number-theoretic nature of communication in quantum spin systems.
\newblock {\em Physical review letters}, 109(5):050502, August 2012.

\bibitem{hou}
H.~Hou, R.~Gu, and M.~Tong.
\newblock Pretty good state transfer on 1-sum of star graphs.
\newblock {\em Open Math.}, 16(1):1483--1489, 2018.

\bibitem{pal6}
H.~Pal.
\newblock More circulant graphs exhibiting pretty good state transfer.
\newblock {\em Discrete Math.}, 341(4):889--895, 2018.

\bibitem{pal7}
H.~Pal.
\newblock Quantum state transfer on a class of circulant graphs.
\newblock {\em Linear and Multilinear Algebra}, 0(0):1--12, 2019.

\bibitem{pal1}
H.~Pal and B.~Bhattacharjya.
\newblock Perfect state transfer on {NEPS} of the path on three vertices.
\newblock {\em Discrete Math.}, 339(2):831--838, 2016.

\bibitem{pal2}
H.~Pal and B.~Bhattacharjya.
\newblock Perfect state transfer on gcd-graphs.
\newblock {\em Linear Multilinear Algebra}, 65(11):2245--2256, 2017.

\bibitem{pal4}
H.~Pal and B.~Bhattacharjya.
\newblock Pretty good state transfer on circulant graphs.
\newblock {\em Electron. J. Combin.}, 24(2):Paper No. 2.23, 13, 2017.

\bibitem{pal5}
H.~Pal and B.~Bhattacharjya.
\newblock Pretty good state transfer on some {NEPS}.
\newblock {\em Discrete Math.}, 340(4):746--752, 2017.

\bibitem{bom1}
C.~M. van Bommel.
\newblock Pretty good state transfer and minimal polynomials.
\newblock {\em arXiv preprint arXiv:2010.06779}, 2020.

\bibitem{vin}
L.~Vinet and A.~Zhedanov.
\newblock Almost perfect state transfer in quantum spin chains.
\newblock {\em Physical Review A}, 86(5):052319, 2012.

\end{thebibliography}

\end{document}